\documentclass{amsart}
\usepackage{amsmath}

\usepackage{amssymb}
\usepackage{amsthm}
\usepackage{amscd,amsbsy}
\usepackage{xcolor}
\usepackage{float}

\newtheorem{theorem}{Theorem}[section]
\newtheorem{lemma}[theorem]{Lemma}

\newtheorem{corollary}[theorem]{Corollary}

\theoremstyle{definition}
\newtheorem{definition}[theorem]{Definition}

\newtheorem{definitions and remarks}[theorem]{Definitions and Remarks}

\theoremstyle{remark}

\newtheorem{remark}[theorem]{Remark}
\newtheorem{remarks}[theorem]{Remarks}

\numberwithin{equation}{section}

\newcommand{\inv}{\mathrm{inv}}

\newcommand{\supp}{\mathrm{supp}\,}

\newcommand{\mon}{\mathrm{mon}\,}

\newcommand{\lex}{\mathrm{lex}}
\newcommand{\Rel}{\mathrm{Rel}}

\newcommand{\al}{{\alpha}}
\newcommand{\be}{{\beta}}

\newcommand{\De}{{\Delta}}
\newcommand{\ga}{{\gamma}}

\newcommand{\la}{{\lambda}}

\newcommand{\Om}{{\Omega}}
\newcommand{\p}{{\partial}}

\newcommand{\Th}{{\Theta}}

\newcommand{\IN}{{\mathbb N}}

\newcommand{\IR}{{\mathbb R}}

\newcommand{\IK}{{\mathbb K}}

\newcommand{\cC}{{\mathcal C}}
\newcommand{\cD}{{\mathcal D}}

\newcommand{\cF}{{\mathcal F}}

\newcommand{\cI}{{\mathcal I}}

\newcommand{\cN}{{\mathcal N}}
\newcommand{\cO}{{\mathcal O}}
\newcommand{\cQ}{{\mathcal Q}}

\newcommand{\cV}{{\mathcal V}}

\newcommand{\uk}{\underline{k}}

\newcommand{\um}{\underline{m}}

\newcommand{\hf}{{\hat f}}

\newcommand{\hPhi}{{\widehat \Phi}}

\newcommand{\llb}{{[\![}}
\newcommand{\rrb}{{]\!]}}

\newcommand{\RN}[1]{%
  \textup{\uppercase\expandafter{\romannumeral#1}}%
}

\begin{document}
\title[Malgrange division by quasianalytic functions]
{Malgrange division by quasianalytic functions}

\author[E.~Bierstone]{Edward Bierstone}
\author[P.D.~Milman]{Pierre D. Milman}
\address{University of Toronto, Department of Mathematics, 40 St. George Street,
Toronto, ON, Canada M5S 2E4}
\email[E.~Bierstone]{bierston@math.toronto.edu}
\email[P.D.~Milman]{milman@math.toronto.edu}
\thanks{Research supported in part by NSERC grants OGP0009070 and OGP0008949}



\begin{abstract}
Quasianalytic classes are classes of $\cC^\infty$ functions that satisfy
the analytic continuation property enjoyed by analytic functions. Two general examples are quasianalytic
Denjoy-Carleman classes (of origin in the analysis of linear partial differential equations) and the class
of $\cC^\infty$ functions that are definable in a polynomially bounded $o$-minimal structure (of origin
in model theory). We prove a generalization to quasianalytic functions of Malgrange's celebrated theorem 
on the division of $\cC^\infty$ by real-analytic functions.
\end{abstract}

\date{\today}
\maketitle
\setcounter{tocdepth}{1}
\tableofcontents

\section{Introduction}\label{sec:intro} 
%Throughout this article, we work with a given quasianalytic class $\cQ$ of $\cC^\infty$ functions. 
A quasianalytic class $\cQ$ associates, to every open
subset $U \subset \IR^n$, a subring $\cQ(U)$ of $\cC^\infty(U)$ which satisfies the basic properties
of $\cC^\infty$ functions together with the following axiom of \emph{quasianalyticity}: if the Taylor
expansion $\hf_a$ of $f\in \cC^\infty(U)$ at a point $a \in U$ vanishes identically, then
$f$ is identically zero in a neighbourhood of $a$ (see Section \ref{sec:quasian} for precise definitions.)
An important special case is $\cQ = \cO$, the class of real-analytic functions, and quasianalyticity is
a generalization of the classical property of analytic continuation. 

Two general examples of quasianalytic classes are Denjoy-Carleman classes (see \S\ref{subsec:DC}), 
and the class of $\cC^\infty$ functions that are definable in a given polynomially bounded $o$-minimal 
structure \cite{Mil}, \cite{RSW}. The former arise in the analysis of linear partial differential equations, 
and the latter in model theory.

\begin{theorem}\label{thm:main} Let $\Phi_1,\ldots,\Phi_q \in \cQ(U)^p$, where $U$ is open in $\IR^n$
and $p\in \IN$. Given $f \in \cC^\infty(U)^p$, we can find $g_1,\ldots,g_q \in \cC^\infty(U)$
such that
\begin{equation}\label{eq:main}
f = \sum_{j=1}^q g_j \Phi_j,
\end{equation} 
if and only if the equation \eqref{eq:main} admits a formal power series solution at every point $a\in U$; i.e.,
$\hf_a = \sum_{j=1}^q G_{j,a} \hPhi_{j,a}$, where each $G_{j,a} \in \cF_a$, $a\in U$. ($\cF_a$ denotes the ring
of formal power series centred at $a$.)
\end{theorem}

The formal condition at every point is, of course, necessary.
Theorem \ref{thm:main} in the special case $\cQ=\cO$ is the celebrated division theorem of Malgrange
(see \cite[Ch.\,VI]{Malg}); the case $p=q=1$ of Malgrange's theorem is due to H\"ormander \cite{Horm}
and {\L}ojasiewicz \cite{Loj}. Malgrange's proof relies on fundamental algebraic
properties of the ring $\cO_a$ of germs of real-analytic functions at a point $a$; in particular, on 
Noetherianity of $\cO_a$ and flatness of $\cO_a$ over $\cF_a$ (as reflected in Oka's theorem on
coherence of the sheaf of relations among analytic functions). These properties are not known for
quasianalytic classes, in general. Our proof of Theorem \ref{thm:main} requires only a weaker topological
version of Noetherianity, which is a consequence of resolution of singularities 
for quasianalytic classes \cite{BMselecta}. Theorem \ref{thm:main}
in the case $p=q=1$ is proved in \cite[Thm.\,6.4]{BMselecta} using resolution of singularities, but we
will not use this special case.

Our proof of Theorem \ref{thm:main} uses Hironaka's elementary formal division algorithm (a generalization
of Euclidean division to division by several functions, presented first in \cite{HiroAnn}; see also \cite{BMPisa})
in the way that we used it to prove Malgrange's division theorem in \cite[\S10]{BMrel}. We will need only
a stratified version of Oka's coherence theorem (``semicoherence'', in the language of \cite{BMann}) 
which, for quasianalytic functions, is a simple consequence
of the formal division algorithm and topological Noetherianity (see Section \ref{sec:param} and Theorem
\ref{thm:strat}). We feel that part of the interest of this article
is that our proof of Theorem \ref{thm:main} in the classical case $\cQ =\cO$ seems the simplest and
most direct approach to Malgrange's division theorem. Resolution of singularities 
is used in the paper to prove several geometric or metric properties of sets defined by quasianalytic
functions, which, in the case $\cQ=\cO$, are well-known consequences of techniques involving Weierstrass
preparation. The reader might want to look at the proof of Theorem \ref{thm:main} in \S\ref{subsec:main}
first, and work backwards from there, as a motivation for the parametrized division techniques developed
in Sections \ref{sec:formaldiv} and \ref{sec:param}.

Noetherianity of the local ring $\cQ_a$ of germs of functions of quasianalytic class $\cQ$ at a point
$a\in \IR^n$ is equivalent to the following variant of Theorem \ref{thm:main}: given $f \in \cQ_a$, 
there exist $g_1,\ldots,g_q \in \cQ_a$ such that $f = \sum_{j=1}^q g_j \Phi_j$ if and only if there is 
a formal solution at $a$. The existence of a solution $g_1,\ldots,g_q \in \cQ_a$ is unknown even under 
the stronger assumption of a formal solution at every point in a neighbourhood of $a$. It seems plausible
that, even under the latter assumption, a quasianalytic solution $g_1,\ldots,g_q$
may necessarily involve a certain loss of regularity (i.e., belong to a larger quasianalytic class $\cQ' \supseteq
\cQ$) depending on $\Phi_1,\ldots,\Phi_q$ and $f$; cf. \cite{BBB}.

\section{Hironaka's division algorithm and formal relations}\label{sec:formaldiv}
We use standard multiindex notation: Let $\IN$ denote the nonnegative integers. If $\al = (\al_1,\ldots,\al_n) \in
\IN^n$, we write $|\al| := \al_1 +\cdots +\al_n$, $\al! := \al_1!\cdots\al_n!$, $x^\al := x_1^{\al_1}\cdots x_n^{\al_n}$,
and $\p^{|\al|} / \p x^{\al} := \p^{\al_1 +\cdots +\al_n} / \p x_1^{\al_1}\cdots \p x_n^{\al_n}$. 
%We write $(i)$ for the multiindex with $1$ in the $i$th place and $0$ elsewhere.

Let $A$ be a ring. Let $A\llb x\rrb = A\llb x_1,\ldots,x_n\rrb$
denote the ring of formal power series in $n$ indeterminates with coefficients in $A$. Given $F(Y) \in A\llb x\rrb^p$,
we write $F(Y) = \sum_{j=1}^p \sum_{\al\in\IN^n} F_{(\al,j)}x^{(\al,j)}$, where each coefficient
$F_{(\al,j)}\in A$, and $x^{(\al,j)}:= (0,\ldots,x^\al,\ldots,0)$ with $x^\al$ in the $j$th place. If $(\al,j) \in
\IN^n\times\{1,\ldots,p\}$ and $\be \in \IN^n$, we define $(\al,j) + \be := (\al +\be,j)$. We will use the
following notation:
\begin{align*}
\supp F &:= \{(\al,j)\in \IN^n\times\{1,\ldots,p\}:\, F_{(\al,j)}\neq 0\},\\
\exp F &:= \min\,\supp F,\\
\mon F &:= F_{\exp F}x^{\exp F},
\end{align*}
where $\min$ is with respect to the total ordering of $\IN^n\times\{1,\ldots,p\}$ given by 
$\lex (|\al|,j,\allowbreak\al) = \lex (|\al|,j,\allowbreak\al_1,\ldots,\al_n)$, 
where $\lex$ denotes lexicographic order. We call $\exp F$
the \emph{initial exponent}, $\mon F$ the \emph{initial monomial} and $F_{\exp F}$ the
\emph{initial coefficient} of $F$. 

\begin{remark}\label{rem:order}
In the following,
instead of the preceding order, we may also use any total ordering of $\IN^n\times\{1,\ldots,p\}$ which is
compatible with addition in the sense that $(\al,j) + \be > (\al,j)$ if $\be \in \IN^n\setminus\{0\}$.
For example, we may use $\lex (L(\al),j,\al)$, where $L$ is a positive linear form 
$L(\al) = \sum_{i=1}^n \la_i\al_i$ (i.e., all $\la_i$ are positive real numbers).
%; in this case, we write $\ord\,F := L(\al)$, where $(\al,j) = \exp F$. 
\end{remark}

\subsection{Hironaka's formal division algorithm \cite{HiroAnn}, \cite[\S6]{BMrel}}\label{subsec:formaldiv}
\begin{theorem}\label{thm:formaldiv}
Let $L$ denote a positive linear form on $\IR^n$, and consider the corresponding ordering of
$\IN^n\times\{1,\ldots,p\}$ (as in Remark \ref{rem:order}).
Let $\Phi_1,\ldots,\Phi_q\in A\llb x\rrb^p$. Set $(\al_i,j_i):=\exp \Phi_i$, $i=1,\ldots,q$, and consider the
partition $\{\De_1,\ldots\De_q,\De\}$ of $\IN^n\times\{1,\ldots,p\}$ given by $\De_1 := (\al_1,j_1) +\IN^n$,
\begin{align*}
\De_i &:= (\al_i,j_i) +\IN^n \setminus \bigcup_{k=1}^{i-1}\De_k,\quad i=2,\ldots,q,\\
\De &:= \IN^n\times\{1,\ldots,p\} \setminus \bigcup_{i=1}^{q}\De_i.
\end{align*}
Suppose that $A$ is an integral domain, and let $\IK$ be the field of fractions of $A$.
Let $S$ denote the multiplicative subset of $A$ generated by the initial coefficients $\Phi_{i,(\al_i,j_i)}$,
and let $B$ denote any subring of $\IK$ containing the localization $S^{-1}A$.
Then, for every $F \in B\llb x\rrb^p$, there exist unique $Q_i = Q_i(F) \in B\llb x\rrb$, $i=1,\ldots,q$, and
$R = R(F) \in B\llb x\rrb^p$ such that
\begin{align*}
F &= \sum_{i=1}^q Q_i\Phi_i + R,\\
(\al_i,j_i) + \supp Q_i \subset &\De_i,\,\,\, i=1,\ldots,q,\,\,\, \text{and }\,\, \supp R \subset \De.
\end{align*}
Moreover,
$$
(\al_i,j_i) +\exp Q_i \geq \exp F,\,\,\, i=1,\ldots,q,\,\,\, \text{and}\,\, \exp R \geq \exp F.
$$

\end{theorem}

\begin{proof} 
Let $F \in B\llb x\rrb^p$. Then $F$ has a unique expression
$$
F = \sum_{i=1}^q Q'_i(F) \Phi_{i,(\al_i,j_i)} x^{(\al_i,j_i)} + R'(F),
$$
where each $Q'_i(F) \in B\llb x\rrb$, $R'(F) \in B\llb x\rrb^p$, 
$(\al_i,j_i) + \supp Q'_i(F) \subset \De_i$ ($i=1,\ldots,q$), $\supp R'(F) \subset \De$.
Clearly, each summand $Q'_i(F) \Phi_{i,(\al_i,j_i)} x^{(\al_i,j_i)}$ and $R'(F)$ has initial exponent
$\geq \exp F$. Set
\begin{align*}
E(F) :&= F - \left(\sum_{i=1}^q Q'_i(F) \Phi_i + R'(F)\right)\\
&= \sum_{i=1}^q Q'_i(F)\left(\Phi_{i,(\al_i,j_i)} x^{(\al_i,j_i)} - \Phi_i\right).
\end{align*}
Then $\exp E(F) > (\al_i,j_i) + \exp Q'_i(F) \geq \exp F$ (for each $i$) .
Define
\begin{align*}
Q_i(F) &:= \sum_{k=0}^\infty Q'_i(E^k(F)),\quad i=1,\ldots, q,\\
R(F) &:= \sum_{k=0}^\infty R'(E^k(F)),
\end{align*}
where $E^0(F) := F$ and $E^{k+1}(F) := E(E^k(F))$, $k\geq 0$. Then all $Q_i(F)$ and $R(F)$
converge with respect to the Krull topology of $\IK\llb x\rrb$, and
$F = \sum Q_i(F) \Phi_i + R(F)$, as required, since, for all $l \in \IN$,
\[
F - \sum_{i=1}^q \sum_{k=0}^l Q_i(E^k(F))\Phi_i - \sum_{k=0}^l R(E^k(F)) = E^{l+1}(F). \qedhere
\]
\end{proof}

\subsection{The diagram of initial exponents \cite[\S6]{BMrel}}\label{subsec:diag} Let $A$ be a ring, and let
$M$ be a submodule
of $A\llb x\rrb^p$. We define the \emph{diagram of initial exponents} $\cN(M) \subset \IN^n\times \{1,\ldots,p\}$ as
$$
\cN(M) := \{\exp F:\, F \in M\setminus \{0\}\}.
$$
Clearly, $\cN(M) + \IN^n = \cN(M)$.
We say that $(\al_0,j_0)\in \IN^n\times \{1,\ldots,p\}$ is a \emph{vertex} of $\cN(M)$ if 
$(\cN(M)\setminus\{(\al_0,j_0)\}) + \IN^n \neq \cN(M)$. It is easy to see that $\cN(M)$ has finitely many vertices.
Let $(\al_i,j_i)$, $i=1,\ldots,q$, denote the vertices of $\cN(M)$. Choose $\Phi_i \in M$ such that
$(\al_i,j_i) = \exp \Phi_i$, $i=1,\ldots,q$.

\begin{corollary}\label{cor:diag}
Assume that $A$ is an integral domain. Then,
using the notation of Theorem \ref{thm:formaldiv}, we have:
\begin{enumerate}
\item\begin{enumerate}
\item $\cN(M) = \bigcup_{i=1}^q \De_i$;
\smallskip\item $\Phi_1,\ldots,\Phi_q$ generate $B\llb x\rrb\cdot M$;
\smallskip\item if $G \in B\llb x\rrb^p$, then $G \in B\llb x\rrb\cdot M$ if and only if $R(G) = 0$.
\end{enumerate}
\medskip\item There exist unique generators $\Psi_i$, $i=1,\ldots,q$, of $S^{-1}A\llb x\rrb\cdot M$
such that
$$
\Psi_i = x^{(\al_i,j_i)} + R_i,\quad \text{where }\, \supp R_i \subset \De.
$$
\end{enumerate}
\end{corollary}

\begin{proof}
(1) (a) is obvious. Let $G \in B\llb x\rrb^p$. Write $G = \sum_{i=1}^q Q_i(G) \Phi_i + R(G)$, according to the
formal division algorithm. Then $G \in B\llb x\rrb\cdot M$ if and only if $R(G) \in M$; i.e., if and 
only if $R(G)=0$ (since $\supp R(G) \bigcap \cN(M) = \emptyset$). (b), (c) follow immediately.

\medskip\noindent
(2) For each $i$, write $x^{(\al_i,j_i)} = \sum_{j=1}^q Q_{ij} \Phi_j - R_i$, according to the division
algorithm. Clearly, we can take $\Psi_i = x^{(\al_i,j_i)} +R_i$, $i=1,\ldots,q$. 
\end{proof}

We call $\Psi_1,\ldots,\Psi_q$ in Corollary \ref{cor:diag} the \emph{standard basis} of $M$ (or of
$S^{-1}A\llb x\rrb\cdot M$).

We totally order the set of diagrams
\begin{equation}\label{eq:diag}
\cD(n,p):= \{\cN \subset \IN^n\times\{1,\ldots,p\}:\, \cN + \IN^n = \cN\},
\end{equation}
as follows: Given $\cN \in \cD(n,p)$, let $\cV(\cN)$ denote the sequence 
obtained by listing the vertices of $\cN$ in increasing order and completing
this list to an infinite sequence by using $\infty$ for all the remaining terms. If $\cN_1, \cN_2 \in \cD(n,p)$
we say that $\cN_1<\cN_2$ if $\cV(\cN_1)< \cV(\cN_2)$ with respect to the lexicographic ordering on the
set of such sequences. Clearly, $\cN_1\leq\cN_2$ if $\cN_1 \supseteq \cN_2$.

\begin{remark}\label{rem:diagdiff}
The condition $\cN + \IN^n = \cN$ implies that the set of elements of $\IR\llb x\rrb^p$ with
supports in the complement of $\cN$ is closed under formal differentiation. This simple property
of the diagram will play an important part in the proof of our main theorem (see \S\ref{subsec:main}).
\end{remark}

Lemma \ref{lem:prep} and Corollary \ref{cor:prep} following will be
needed in \S\ref{subsec:paramrel}. 

\begin{lemma}\label{lem:prep}
Let $A$ denote an integral domain, and let $\IK$ be the field
of fractions of $A$. Let $\Phi_1,\ldots,\Phi_q \in A\llb x\rrb^p$. Let $M\subset A\llb x\rrb^p$
and $M' \subset \IK\llb x\rrb^p$ denote the submodules generated by $\Phi_1,\ldots,\Phi_q$.
Then $\cN(M) = \cN(M')$. 
\end{lemma}

\begin{proof}
Clearly, $\cN(M) \subset \cN(M')$. Given $(\al,j) \in \cN(M')$, we can choose
$p_i(x) \in \IK[x]$, $i=1,\ldots,q$, such that $\exp\left(\sum_{i=1}^q p_1\Phi_i\right) = (\al,j)$; clearing
the denominators of the coefficients of the $p_i(x)$, we get $\Phi \in M$ such that $\exp \Phi = (\al,j)$.
Hence $\cN(M') \subset \cN(M)$.
\end{proof}

\begin{remarks}\label{rem:mult,nak}
(1) By the formal division theorem, submodules $E\subset F$ of $\IK\llb x\rrb^p$ coincide if and
only if they have the same diagram. Let $N\subset A\llb x\rrb^p$ be a submodule of $M$. It follows from
Lemma \ref{lem:prep} that $\cN(N) = \cN(M)$ if and only if $M,\,N$ generate the same submodule of $\IK\llb x\rrb^p$.
The construction of standard bases shows, in fact, that there is a finitely generated multiplicative subset $S$
of $A$ such that $\cN(N) = \cN(M)$ if and only if $M,\,N$ generate the same 
submodule of $S^{-1}A\llb x\rrb^p$.

\medskip\noindent
(2) Let $R$ denote a local ring, with maximal ideal $\um$ and residue field $\uk = R/\um$. Let
$E$ be a finitely generated $R$-module. If $g_R(E)$ denotes
the minimal number of generators of $E$, then
$$
g_R(E) = \dim_{\uk} E\otimes_R \uk = \dim_{\uk} E/\um\cdot E,
$$
(by Nakayama's lemma).
\end{remarks}

\begin{corollary}\label{cor:prep}
Suppose that $\Phi_1,\ldots,\Phi_m$ is a smallest subset of $\Phi_1,\ldots,\Phi_q$ such that
$\cN(N) = \cN(M)$, where $N$ is the submodule of $M$ generated by $\Phi_1,\ldots,\Phi_m$.
Then
$$
m = \dim_{\IK}N\otimes_{\IK\llb x\rrb} \IK = \dim_{\IK}N/\um\cdot N,
$$
where $\um$ denotes the maximal ideal of $\IK\llb x\rrb$.
\end{corollary}

\begin{proof}
This is a consequence of Lemma \ref{lem:prep} and Remarks \ref{rem:mult,nak}.
\end{proof}

\subsection{The module of formal relations \cite[Ch.\,6]{BMPisa}}\label{subsec:rel} 
We continue to use the notation of Theorem \ref{thm:formaldiv}.
Let $F_1,\ldots,F_q \in B\llb x\rrb^p$. The \emph{module of relations} among $F_1,\ldots,F_q$ denotes
the submodule $\Rel = \Rel(F_1,\ldots,F_q)$ of $B\llb x\rrb^q$ defined as
$$
\Rel := \{H = (H_1,\ldots,H_q) \in B\llb x\rrb^q:\, \sum_{i=1}^q H_i F_i = 0\}.
$$

Let $A$ be an integral domain, and let $M$ be a submodule of $A\llb x\rrb^p$. 
Let $\Psi_1,\ldots,\Psi_q \allowbreak\in S^{-1}A\llb x\rrb^p \subset B\llb x\rrb^p$
denote the standard basis of $M$, and let $\Rel \subset S^{-1}A\llb x\rrb^p \subset B\llb x\rrb^q$ 
be the module of relations among
$\Psi_1,\ldots,\Psi_q$. Consider the partition $\{\De_i\}$ of the diagram of initial exponents $\cN(M)$ given in
Theorem \ref{thm:formaldiv}. Each $\De_i$ has the form
$$
\De_i = (\al_i,j_i) + \Box_i,\,\,\, \text{where}\,\, \Box_i \subset \IN^n.
$$
We define
$$
\cN := \{(\ga,i) \in \IN^n\times\{1,\ldots,q\}:\, \ga \notin \Box_i\};
$$
i.e., each $\cN \bigcap \left(\IN^n\times\{i\}\right)$ is the complement of $\Box_i\times\{i\}$.
Clearly, $\cN + \IN^n = \cN$. We will show that $\cN$ is the diagram of initial exponents $\cN(\Rel)$
of $\Rel$, for a suitable ordering of $\IN^n\times\{1,\ldots,q\}$.

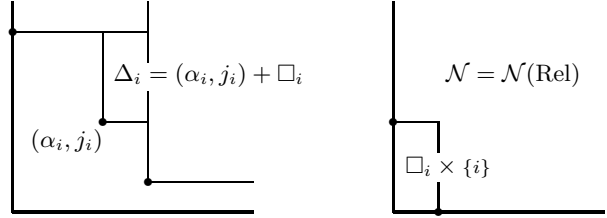
\setlength{\unitlength}{.2cm}
\begin{figure}[H]
\minipage[b]{0.15\textwidth}
\begin{picture}(7,14)
\end{picture}
\endminipage
\minipage[b]{0.4\textwidth}\begin{small}
\begin{picture}(16,14)
\thicklines
\put(0,0){\line(1,0){16}}
\put(0,0){\line(0,1){14}}
\thinlines
\put(0,12){\line(1,0){9}}\put(0,12){\circle*{.5}}
\put(9,2){\line(1,0){7}}
\put(9,2){\line(0,1){6}}\put(9,2){\circle*{.5}}
\put(9,10){\line(0,1){4}}
\put(6,6){\line(1,0){3}}
\put(6,6){\line(0,1){6}}\put(6,6){\circle*{.5}}
%\put(0.2,4.8){\makebox(0,0)[lt]{$\al_1$}}
%\put(4,.8){\makebox(0,0)[rt]{$\al_2$}}
\put(6,5.5){\makebox(0,0)[rt]{$(\al_i,j_i)$}}
%\put(3.5,6){\makebox(0,0){$\D_1$}}
%\put(5.5,3){\makebox(0,0){$\D_2$}}
\put(6.7,8.3){\makebox(0,0)[lb]{$\De_i = (\al_i,j_i) + \Box_i$}}
%\put(.75,.75){\makebox(0,0){$\De$}}
\end{picture}\end{small}
\endminipage
\minipage[b]{0.4\textwidth}\begin{small}
\begin{picture}(14,14)
\thicklines
\put(0,0){\line(1,0){14}}
\put(0,0){\line(0,1){14}}
\thinlines
\put(3,0){\line(0,1){2}}\put(3,0){\circle*{.5}}
\put(3,4){\line(0,1){2}}
\put(0,6){\line(1,0){3}}\put(0,6){\circle*{.5}}
\put(0.8,2.3){\makebox(0,0)[lb]{$\Box_i \times \mbox{\scriptsize{$\{i\}$}}$}}
\put(3.5,8.3){\makebox(0,0)[lb]{$\cN = \cN(\Rel)$}}
\end{picture}\end{small}
\endminipage
\end{figure}

Let $(\ga_k,i_k)$, $k=1,\ldots,s$, denote the vertices of $\cN$. By the 
formal division algorithm (Theorem \ref{thm:formaldiv}), for all $k=1,\ldots,s$,
$$
x^{\ga_k}\Psi_{i_k} = \sum_{j=1}^q Q_{kj}\Psi_j,\quad \text{where each }\, Q_{kj}\in B\llb x\rrb,\,\, \supp Q_{kj} \subset \Box_j.
$$
For each $k=1,\ldots,s$, set
\begin{equation}\label{eq:rel}
P_k := x^{(\ga_k,i_k)} - Q_k\, \in \Rel,\quad \text{where }\, Q_k := (Q_{k1},\ldots,Q_{kq}).
\end{equation}

\begin{theorem}\label{thm:rel}
Let $(\ga_k,i_k)$, $k=1,\ldots,s$, denote the vertices of $\cN$, and define $P_1,\ldots,P_s$
as in \eqref{eq:rel}. Then, for a suitable total ordering of $\IN^n\times\{1,\ldots,q\}$, $\cN = \cN(\Rel)$
and $P_1,\ldots,P_s$ is the standard basis of $\Rel$.
\end{theorem}

\begin{remark}\label{rem:rel}
The diagram of initial exponents $\cN(M)$ is defined using the ordering of $\IN^n\times \{1,\ldots,p\}$ given by
$\lex (L(\al),j,\al)$, where $L$ is a positive linear form, as in Remark \ref{rem:order}. The diagram $\cN = \cN(\Rel)$
of Theorem \ref{thm:rel} will be defined in the proof following using a somewhat different ordering of $\IN^n\times \{1,\ldots,q\}$
depending on the same linear form $L$ and also on the vertices of $\cN(M)$.
\end{remark}

\begin{proof}[Proof of Theorem \ref{thm:rel}]
Given $\al \in \IN^n$ and $i=1,\ldots,q$, we can write
\begin{equation}\label{eq:rel1}
x^\al \Psi_i = \sum_{j=1}^q Q_j \Psi_j,\quad \text{where }\, \supp Q_j \subset \Box_j,
\end{equation}
by the formal division algorithm. Set
$$
P^{\al}_i = x^{(\al,i)} - Q,\quad \text{where }\, Q = (Q_1,\ldots,Q_q).
$$

Suppose that $(\al,i) \in \cN$. Then $(\al +\al_i,j_i) \in \De_h$, for a unique $h=h(\al,i)<i$, and
$x^\al x^{(\al_i,j_i)} = x^{\be(\al,i)} x^{(\al_h,j_h)}$, where $\be(\al,i) \in \Box_h$ (of course, $j_h=j_i$). We can rewrite
\eqref{eq:rel1} as
\begin{equation*}
x^\al \Psi_i = x^{\be(\al,i)} \Psi_{h(\al,i)} + \sum_{h=1}^q Q'_h \Psi_h,
\end{equation*}
where
\begin{equation*}
Q'_h = \begin{cases}
Q_h, &h\neq h(\al,i)\\
Q_{h(\al,i)}-x^{\be(\al,i)}, &h = h(\al,i).
\end{cases}
\end{equation*}
By Theorem \ref{thm:formaldiv},
\begin{equation}\label{eq:rel2}
\exp x^\al \Psi_i = \exp x^{\be(\al,i)} \Psi_{h(\al,i)} < \exp Q'_h \Psi_h,\quad h=1,\ldots,q.
\end{equation}

Consider the total ordering of $\IN^n\times\{1,\ldots,q\}$ given by $(\al, i) \mapsto
\lex (L(\al)+L(\al_i), \al+\al_i, -i)$. Clearly, $(\al,j) + \be > (\al,j)$ with this order, if $\be \in \IN^n\setminus\{0\}$;
cf. Remark \ref{rem:order}.

If $(\al,i) \in \cN$, then $\mon P^\al_i = x^{(\al,i)}$, by \eqref{eq:rel2}; therefore, $\cN \subset \cN(\Rel)$
(where $\mon$ and $\cN(\Rel)$ are defined with respect to the preceding order).
To get the opposite inclusion, consider $H = (H_1,\ldots,H_q) \in \Rel$ and divide $H$ by $P_1,\ldots,P_s$
according to the formal division algorithm:
$$
H = \sum_{k=1}^s \xi_kP_k + T,\quad \text{where } T=(T_1,\ldots,T_q),\, \supp T_i \subset \Box_i.
$$
Since $\sum T_i\Psi_i = 0$ and each $\supp T_i \subset \Box_i$, it follows that $T=0$, by uniqueness
in the division algorithm; hence $\cN(\Rel) \subset \cN$. 

Therefore, $\cN = \cN(\Rel)$, and $P_1,\ldots,P_s$ is the standard basis of $\Rel$.
\end{proof}

\section{Quasianalytic classes}\label{sec:quasian}
We consider a class of functions $\cQ$ given by the association, to every 
open subset $U\subset \IR^n$, of a subalgebra $\cQ(U)$ of $\cC^\infty (U)$ containing
the restrictions to $U$ of polynomial functions on $\IR^n$, and closed under composition 
with a $\cQ$-mapping (i.e., a mapping whose components belong to $\cQ$). 

Such a class determines a sheaf of local $\IR$-algebras of $\cC^\infty$ functions on $\IR^n$,
for each $n$, that we also denote $\cQ$.

\begin{definition}[quasianalytic classes]\label{def:quasian}
We say that $\cQ$ is \emph{quasianalytic} if it satisfies the following three axioms:

\begin{enumerate}
\item \emph{Closure under division by a coordinate.} If $f \in \cQ(U)$ and
$$
f(x_1,\dots, x_{i-1}, a, x_{i+1},\ldots, x_n) = 0,
$$
where $a \in \IR$,  then $f(x) = (x_i - a)h(x),$ where $h \in \cQ(U)$.

\smallskip
\item \emph{Closure under inverse.} Let $\varphi : U \to V$
denote a $\cQ$-mapping between open subsets $U$, $V$ of $\IR^n$.
Let $a \in  U$ and suppose that the Jacobian matrix
$$
\frac{\partial \varphi}{\partial x} (a) := \frac{\partial
(\varphi_1,\ldots, \varphi_n) }{\partial (x_1,\ldots, x_n)}(a)
$$
is invertible. Then there are neighbourhoods $U'$ of $a$ and $V'$ of 
$b := \varphi(a)$, and a $\cQ$-mapping  $\psi: V' \to U'$ such that
$\psi(b) = a$ and $\psi\circ \varphi$  is the identity mapping of
$U '$.

\smallskip
\item \emph{Quasianalyticity.} If $f \in \cQ(U)$ has Taylor expansion zero
at $a \in U$, then $f$ is identically zero near $a$.
\end{enumerate}
\end{definition}

\begin{remarks}\label{rem:axioms} (1)\, Axiom \ref{def:quasian}(1) implies that, 
if $f \in \cQ(U)$, then all partial derivatives of $f$ belong to $\cQ(U)$. 

\smallskip\noindent
(2)\, Axiom \ref{def:quasian}(2) is equivalent to the property that the implicit function theorem holds for functions of 
class $\cQ$.  It implies that the reciprocal of a nonvanishing function of class $\cQ$ is also of class $\cQ$.
\end{remarks} 

The elements of a quasianalytic class $\cQ$ will be called \emph{quasianalytic functions}. 
A category of manifolds and mappings of class $\cQ$ can be defined in a standard way. The category 
of $\cQ$-manifolds is closed under blowing up with centre a $\cQ$-submanifold \cite{BMselecta}.

Resolution of singularities holds in a quasianalytic class $\cQ$ \cite{BMinv}, \cite{BMselecta}. 
Resolution of singularities of an ideal generated by functions of class $\cQ$ can be used to show that sets
defined using such functions enjoy many of the important geometric properities of
real-analytic or subanalytic sets (see \S\ref{subsec:geom} below). In particular, a quasianalytic 
class $\cQ$ determines a polynomially-bounded $o$-minimal structure \cite{RSW}.

\subsection{Quasianalytic Denjoy-Carleman classes}\label{subsec:DC}

\begin{definition}[Denjoy-Carleman classes]\label{def:DC}
Let $M = (M_k)_{k\in \IN}$ denote a sequence of positive real numbers which is \emph{logarithmically
convex}; i.e., the sequence $(M_{k+1} / M_k)$ is nondecreasing. A \emph{Denjoy-Carleman
class} $\cQ = \cQ_M$ is a class of $\cC^\infty$ functions determined by the following condition: A function 
$f \in \cC^\infty(U)$ (where $U$ is open in $\IR^n$) is of class $\cQ_M$ if, for every compact subset $K$ of $U$,
there exist constants $A,\,B > 0$ such that
\begin{equation}\label{eq:DC}
\left|\frac{\p^{|\al|}f}{\p x^{\al}}\right| \leq A B^{|\al|} \al! M_{|\al|}
\end{equation}
on $K$, for every $\al \in \IN^n$.
\end{definition}

The logarithmic convexity assumption implies that
$M_jM_k \leq M_0M_{j+k}$, for all $j,k$, and that
the sequence $(M_k^{1/k})$ is nondecreasing.
The first of these conditions guarantees that $\cQ_M(U)$ is a ring, and 
the second that $\cQ_M(U)$ contains the ring $\cO(U)$ of real-analytic functions on $U$,
for every open $U\subset \IR^n$.
(If $M_k=1$, for all $k$, then $\cQ_M = \cO$.)

A Denjoy-Carleman class $\cQ_M$ is a quasianalytic class in the sense of Definition \ref{def:quasian}
if the sequence
$M = (M_k)_{k\in \IN}$ satisfies the following two assumptions in addition to those
of Definition \ref{def:DC}. 
\begin{enumerate}
\item $\displaystyle{\sup \left(\frac{M_{k+1}}{M_k}\right)^{1/k} < \infty}$.

\smallskip
\item $\displaystyle{\sum_{k=0}^\infty\frac{M_k}{(k+1)M_{k+1}} = \infty}$.
\end{enumerate}

It is easy to see that the assumption (1) implies that $\cQ_M$ is closed under differentiation.
The converse of this statement is due to S. Mandelbrojt \cite{Mandel}. In a Denjoy-Carleman class
$\cQ_M$, closure under differentiation is equivalent to the axiom \ref{def:quasian}(1) of closure under division by a
coordinate---the converse of Remark \ref{rem:axioms}(1) is a simple consequence of the fundamental
theorem of calculus; see \cite[\S3.1]{BBB}.

According to the Denjoy-Carleman theorem, the class $\cQ_M$ is quasianalytic (axiom \ref{def:quasian}(3)) if and only if the assumption (2) holds \cite[Thm.\,1.3.8]{Hormbook}.

Closure of the class $\cQ_M$ under composition is due to Roumieu \cite{Rou} and closure under
inverse to Komatsu \cite{Kom}; see \cite{BMselecta} for simple proofs. The assumptions (1)--(3)
above thus guarantee that $\cQ_M$ is a quasianalytic class.

\subsection{Geometry of quasianalytic classes}\label{subsec:geom} 
Let $\cQ$ denote a quasianalytic class, and let $U$ be an open subset of $\IR^n$. A closed subset
$X$ of $U$ will be called \emph{quasianalytic} (of class $\cQ$) if every point of $X$ admits a neighbourhood
$V$ in $U$ such that $X\cap V$ is the zero set of a finite family of elements of $\cQ(V)$ (we will also call
$X$ a closed \emph{$\cQ$-subset} of $U$). A point $a\in X$ will be called a \emph{smooth point} of $X$ if $X$ is 
a manifold of class $\cQ$ in some neighbourhood of $a$. Let $a\in U$. A \emph{germ of a closed $\cQ$-set} at a point
$a \in U$ means a germ at $a$ of a closed $\cQ$-subset of some neighbourhood
of $a$. 

The following two lemmas are consequences of resolution of singularities or the techniques
involved \cite{BMselecta}.

\begin{lemma}[{topological Noetherianity \cite[Thm.\,6.1]{BMselecta}}]\label{lem:noeth}
Let $a\in U$. Then any decreasing sequence $X_1 \supset X_2 \supset\cdots$ of germs of closed $\cQ$-sets at $a$
stabilizes (i.e., there exists $k$ such that $X_j=X_k$, for all $j\geq k$).
\end{lemma}

\begin{lemma}\label{lem:smooth}
Let $X$ denote a closed $\cQ$-subset of $U$, and let $a\in X$. Then there is a neighbourhood
$V$ of $a$ in $U$ and a finite filtration
$$
X\cap V = X_0 \supset X_1 \supset \cdots \supset X_{t+1} = \emptyset,
$$
where, for each $k=0,\ldots,t$, $X_k$ is a closed $\cQ$-subset of $V$ and
$X_k\backslash X_{k+1}$ is smooth.
\end{lemma}

\begin{proof}
By Lemma \ref{lem:noeth}, it is enough to find a (relatively compact) neighbourhood $V$ of $a$ and a closed
$\cQ$-subset $Y$ of $X\cap V$ such that $(X\cap V)\backslash Y$ is smooth. We can choose $V$ so
that $X\cap V$ is the zero set of a finite collection of quasianalytic functions $f_1,\ldots,f_q \in \cQ(V)$.
Let $\cI$ denote the sheaf of ideals of $\cQ|_V$ generated by $f_1,\ldots,f_q$. The assertion of the lemma
follows from properties of the desingularization invariant $\inv_{\cI}$, which determines an algorithm 
for resolution of singularities of $\cI$ \cite{BMfunct}: The invariant $\inv_{\cI}$ is upper-semicontinuous
with respect to the $\cQ$-Zariski topology (i.e., the topology whose closed sets are the closed
$\cQ$-sets), and the minimum points of $\inv_{\cI}$ on $X\cap V$ are smooth points. 
\end{proof}

\begin{remark}\label{rem:smooth}
The proof above requires only properties of $\inv_{\cI}$ in ``year zero'' (i.e., before we start
blowing up), which is much simpler than the invariant defined over a sequence of blowings-up,
in general.
It is a good exercise to unwind the construction of $\inv_{\cI}$ (as presented, for example, in
\cite{BMfunct}) to describe the subset $Y$ explicitly as a finite union of closed $\cQ$-subsets, 
each obtained as the zero set of finitely many functions that are polynomials in the derivatives 
of $f_1,\ldots,f_q$.
\end{remark}

\begin{corollary}\label{cor:smooth1}
The set of smooth points of a closed $\cQ$-subset $X$ of $U$ is dense in $X$.
\end{corollary}

\begin{proof}
Using the notation of Lemma \ref{lem:smooth}, let $\Om_k := X_k\backslash X_{k+1}$,
$k=0,\ldots,t$. Then $\Om_0 \cup \bigcup_{k=1}^t \Om_k\backslash \overline{\Om}_{k-1}$ is a dense
open subset of $X\cap V$ consisting of smooth points.
\end{proof}

\begin{lemma}\label{lem:smooth2}
Let $X$ be a closed $\cQ$-subset of $U$, and let $a\in X$. Then there is a neighbourhood
$V$ of $a$ in $U$ and a dense open subset $\Om$ of $X\cap V$ such that $\Om$ is smooth and
has only finitely many connected components, each adherent to $a$.
\end{lemma}

\begin{proof}
This is a consequence of resolution of singularities \cite[Thm.\,5.9]{BMselecta} or the simpler
rectilinearization theorem \cite[Cor.\,5.13]{BMselecta} (or, alternatively, of the fact that $\cQ$
determines a polynomially-bounded $o$-minimal structure \cite{RSW}).
\end{proof}

Let $a\in U$. A germ of a closed $\cQ$-set at $a$ is \emph{irreducible} if it cannot be written as the union of two proper subgerms of closed $\cQ$-sets.
It follows from Lemma \ref{lem:noeth} that any germ of a closed $\cQ$-set has finitely many irreducible
components (in the same sense as for $\cQ =\cO$). 

Let $X$ denote a closed $\cQ$-subset $X$. We let $\cQ(X)$ denote the ring of restrictions to 
$X$ of quasianalytic functions defined in neighbourhoods of $X$. If $Z$ is a closed $\cQ$-subset
of $X$, let $\cQ(X,Z)$ denote the ring of quotients of elements of $\cQ(X)$ with denominators
vanishing nowhere in $X\backslash Z$. The following is a consequence of Lemma \ref{lem:smooth2}.

\begin{corollary}\label{cor:irred}
Let $Z\subset X$ denote closed $\cQ$-subsets of $U$ (perhaps $Z=\emptyset$), and let $a\in X$.
Suppose that the germ of $X$ at $a$ is irreducible. Then, by shrinking $U$ to an appropriate
neighbourhood of $a$, we can assume that $\cQ(X,Z)$ is an integral domain.
\end{corollary}

\section{Quasianalytic families of formal power series}\label{sec:param}

\subsection{Parametrized families of formal power series}\label{subsec:param}
Recall that, if $a \in \IR^n$, then $\cF_a$ denotes the ring of formal power series centred at $a$. We
write $\cF_a = \IR\llb x\rrb = \IR\llb x_1,\ldots,x_n\rrb$; i.e., the formal Taylor expansion $\hf_a$ at $a$ of a $\cC^\infty$
function $f(x_1,\ldots,x_n)$ is written $\hf_a(x) = \sum_{\al\in\IN^n}(\p^{|\al|}f/\p x^\al)(a)x^\al/\al!$.

Consider $Z\subset X \subset U$, where $U$ is an open subset of $\IR^n$ and $X,\,Z$ are closed $\cQ$-subsets of $U$.
Set $A := \cQ(X,Z)$. Let $\Phi_1,\ldots,\Phi_q \in A\llb x\rrb^p$; i.e.,
\begin{equation}\label{eq:phi}
\Phi_i (x) = \sum_{j=1}^q\sum_{\al\in\IN^n} \Phi_{i,(\al,j)} x^{(\al,j)},\quad i=1,\ldots,q,
\end{equation}
where each coefficient $\Phi_{i,(\al,j)}\in A$. For each $i=1,\ldots,q$ and $a\in X\setminus Z$, let 
$\Phi_i (a,x) \in \IR\llb x\rrb^p$ denote the power series obtained by evaluating the coefficients
of $\Phi_i$ at $a$; i.e., 
\begin{equation}\label{eq:phieval}
\Phi_i (a,x) = \sum_{j=1}^q\sum_{\al\in\IN^n} \Phi_{i,(\al,j)}(a) x^{(\al,j)}.
\end{equation}

Let $M$ denote the submodule of $A\llb x\rrb^p$ generated by $\Phi_1,\ldots,\Phi_q$, and, for each $a\in 
X\setminus Z$, let $M_a$ denote the submodule of $\IR\llb x\rrb^p$ generated by
$\Phi_1 (a,x),\ldots,\Phi_q (a,x)$. Consider the diagrams of initial exponents
$$
\cN = \cN(M),\,\, \text{and }\,\, \cN_a = \cN(M_a),\,\, a \in X\setminus Z
$$
(using the ordering of $\IN^n\times\{1,\ldots,p\}$ given by $\lex(L(\al),j,\al)$, where $L$ is a positive
linear form, as in Remark \ref{rem:order}). 

The preceding notation will be fixed throughout this section.

\begin{theorem}[semicontinuity of the diagram of initial exponents]\label{thm:param}
\hspace{1cm}
\begin{enumerate}
\item 
If $K\subset X$ is compact, then there are only finitely many values of $\cN_a \in \cD(n,p)$, 
$a \in X\setminus Z$ (see \eqref{eq:diag}).

\smallskip
\item
For each $a_0 \in X\setminus Z$, 
$$
Z \,\bigcup \,\{a \in X\setminus Z:\, \cN_a \geq \cN_{a_0}\}
$$
is a closed $\cQ$-subset of $X$.
\end{enumerate}
\end{theorem}

Theorem \ref{thm:param} is given in \cite[Ch.\,5]{BMPisa} in the case $\cQ=\cO$. If $Z=\emptyset$,
then the conclusions of the theorem mean that $a \mapsto \cN_a \in \cD(n,p)$ is ``$\cQ$-Zariski 
upper-semicontinuous'' on $X$. Theorem \ref{thm:param} is a consequence of the following two
lemmas.

\begin{lemma}\label{lem:param1}
For all $a \in X\setminus Z$, $\cN \leq \cN_a$.
\end{lemma}

\begin{proof}
Let $a \in X\setminus Z$. Let $(\al_i,j_i)$, $i=1,\ldots,s$, and $(\be_i,k_i)$, $i=1,\ldots,t$, denote
the vertices of $\cN_a$ and $\cN$ (respectively), in each case indexed in increasing order.

Consider $F \in M_a$ such that $\exp F = (\al_1,j_1)$, say,
$$
F(x) = \sum_{l=1}^q c_l(x) \Phi_l(a,x),\quad \text{where each }\, c_l(x) \in \IR\llb x\rrb,
$$
and set \hspace{.85cm} $\displaystyle{G(x) := \sum_{l=1}^q c_l(x) \Phi_l(x) \in M.}$

\smallskip
\noindent
Then $\exp G \leq (\al_1,j_1)$, since the coefficient of $x^{(\al_1,j_1)}$ is nonzero.
Therefore, $(\be_1,k_1) \leq \exp G \leq (\al_1,j_1)$; in particular, if $(\be_1,k_1) = (\al_1,j_1)$,
then $\exp G = (\al_1,j_1)$.

Now suppose that, for all $i=1,\ldots,r\leq s$, $(\be_i,k_i) = (\al_i,j_i)$ (in particular, $r\leq t$)
and there exists $G_i = G_i(x) \in M$ such that $\exp G_i = (\al_i,j_i) = G_i(a,x)$. If $r=s$, we
are done. Suppose that $r<s$. Consider
$$
F(x) = \sum_{l=1}^q c_l(x) \Phi_l(a,x) \in M_a, \quad \text{where } \exp F = (\al_{r+1},j_{r+1}),
$$
and set \hspace{.16cm} $\displaystyle{G(x) := \sum_{l=1}^q c_l(x) \Phi_l(x) \in M.}$

\smallskip
\noindent
Then $\exp G \leq (\al_{r+1},j_{r+1})$. If $\exp G = (\al_{r+1},j_{r+1})$, then $r< t$ and
$(\be_{r+1},k_{r+1}) \leq \exp G = (\al_{r+1},j_{r+1})$. On the other hand, if $\exp G < (\al_{r+1},j_{r+1})$, then\\
either\, $\displaystyle{\exp G \notin \bigcup_{i=1}^r \left((\al_i,j_i) + \IN^n\right)}$, so that $r<t$ and $(\be_{r+1},k_{r+1}) 
< (\al_{r+1},j_{r+1})$,\\
\phantom{eith}or\, $\displaystyle{\exp G \in \bigcup_{i=1}^r \left((\al_i,j_i) + \IN^n\right)}$.\\
In the latter case,
$\exp G = (\al_i+\ga,j_i)$, for some $i=1,\ldots,r$ and $\ga\in \IN^n$. Then $\mon G = G_{(\al_i+\ga,j_i)}x^{(\al_i+\ga,j_i)}$,
where $G_{(\al_i+\ga,j_i)}(a) = 0$ since $\exp G < (\al_{r+1},j_{r+1}) = \exp G(a,x)$. But $\mon G_i
= G_{i,(\al_i,j_i)}x^{(\al_i,j_i)}$, where $G_{i,(\al_i,j_i)}(a)\neq0$. Let
$$
G'(x) := G_{i,(\al_i,j_i)}G(x) - G_{(\al_i+\ga,j_i)}x^\ga G_i(x).
$$
Then $\exp G'(a,x) = (\al_{r+1},j_{r+1})$ and $\exp G < \exp G' \leq (\al_{r+1},j_{r+1})$. After finitely
many such steps, the latter $\leq$ becomes $=$. The lemma therefore follows, by induction on $r$.
\end{proof}

\begin{lemma}\label{lem:param2}
Let $a_0 \in X$, and assume that the germ of $X$ at $a_0$ is irreducible. Then, after shrinking $U$ to a suitable
neighbourhood of $a_0$, there is a proper closed $\cQ$-subset $Y$ of $X$
containing $Z$, such that
\begin{enumerate}
\item  
$\cN_a = \cN$, for all $a \in X\setminus Y$;
\smallskip
\item for every vertex $(\al,j)$ of $\cN$,
there exists $G \in M$ such that
$$
\exp G = (\al,j) = \exp G(a,x), \quad \text{for all }\, a \in X\setminus Y.
$$
\end{enumerate}
\end{lemma}

\begin{proof}
Let $(\be_i,k_i)$, $i=1,\ldots,t$, denote the vertices of $\cN$. For each $i=1,\ldots,t$, take 
$$
G_i = \sum_{j=1}^q\sum_{\al\in\IN^n} G_{i,(\al,j)} x^{(\al,j)} \in M, \quad  \text{such that } \exp G_i = (\be_i,k_i). 
$$
Put
$$
Y := Z\, \cup\, \bigcup_{i=1}^t \left\{a \in X\setminus Z:\, G_{i,(\al,j)}(a) = 0\right\}.
$$
By Corollary \ref{cor:irred}, we can assume that $\cQ(X,Z)$ is an integral domain; therefore,
$Y$ is a proper closed subset of $X$. If $a \in X\setminus Y$, then
$G_i(a,x) \in M_a$ and $\exp G_i(a,x) = (\be_i,k_i)$; therefore, $\cN \subset \cN_a$, so that $\cN_a \leq \cN$.
By Lemma \ref{lem:param1}, $\cN_a = \cN$.
\end{proof}

In order to use the results above in an efficient way, it will be convenient to have a stronger version
of Lemma \ref{lem:param2}. Let $a_0 \in X$. By shrinking $U$ to a suitable neighbourhood of $a_0$,
we can assume that $X=\bigcup_{l=1}^t X_l$, where, for each $l$, $X_l$ is a proper closed $\cQ$-subset
of $U$ and the germ of $X_l$ at $a_0$ is irreducible (none contained in another). For each $l$, let
$A_l := \cQ(X_l,Z_l)$, where $Z_l:=Z\cap X_l$, let $M_l$ denote the submodule of $A_l\llb x\rrb^p$ induced
by $M$ (i.e., by the given $\Phi_1,\ldots,\Phi_q$), and set $\cN_l:=\cN(M_l)$. Condition (1) in the
following corollary is used in Section \ref{sec:proof}.

\begin{corollary}\label{cor:irredstrat}
After shrinking $U$ to a suitable neighbourhood of $a_0$, for each $l=1,\ldots,t$:
\begin{enumerate}
\item
there is a proper closed $\cQ$-subset $Y_l$ of $X_l$ containing $Z_l$, such that 
\smallskip
\begin{enumerate}
\item
$\cN_a = \cN_l$,
for all $a\in X_l\backslash Y_l$,
\smallskip
\item
for every vertex $(\al,j)$ of $\cN_l$, there exists $G\in M$ such that
$$
\exp G^l = (\al,j) = \exp G(a,x),
$$
for all $a\in X_l\backslash Y_l$, where $G^l$ denotes the element of $M_l$ induced by $G$;
\end{enumerate}
\smallskip
\item
$\cN \leq \cN_l$.
\end{enumerate}
\end{corollary}

\begin{proof}
By Lemma \ref{lem:param2}, we can assume that, for each $l$, there exists $Y'_l \subset X_l$
such that the conclusions of Lemma \ref{lem:param2} hold with respect to $Y'_l$, $M_l$ and $\cN_l$.

Consider any fixed $l$. We will show that, for every vertex $(\al,j)$ of $\cN_l$, there exists $G\in M$ such
that $\exp G^l = (\al,j)$, where $G^l$ denotes the element of $M_l \subset A_l\llb x\rrb^p$ induced by $G$.
The existence of $Y_l$ with properties (1)(a),\,(b) then follows as in the proof of Lemma  \ref{lem:param2}.
Condition (2) also then follows, by mimicking the proof of Lemma \ref{lem:param1}.

Let $(\al_i,j_i)$, $i=1,\ldots,s$, denote the vertices of $\cN_l$, indexed in increasing order. Consider
a point $a\in X_l\backslash Y'_l$. There exist polynomials $\xi_k(x)$, $k=1,\ldots,q$, such that
$\exp\left(\sum_{k=1}^q \xi_k(a+x) \Phi_k(a,x)\right) = (\al_1,j_1)$. Let $G_1 := \sum_{k=1}^q \eta_k\Phi_k 
\in A\llb x\rrb^p$, where $\eta_k(a,x) = \xi_k(a+x)\in A\llb x\rrb$, $k=1,\ldots,q$. Then $G_1\in M$. Let $G_1^l \in M_l$
denote the element induced by $G_1$. Then $\exp G_1^l = (\al_1,j_1)$, since $(\al_1,j_1)$ is the smallest
element of $\cN_l$.

We now argue as in the proof of Lemma \ref{lem:param1}. Suppose that $t>1$ and that, for all $h=1,\ldots,r$ (where
$1\leq r < t$), there exists $G_h\in M$ such that $\exp G_h^l = (\al_h,j_h)$, where $G_h^l \in M_l$ is the element
induced by $G_h$. As above, there exists $G\in M$ such that $\exp G(a,x) = (\al_{r+1},j_{r+1})$. Let
$G^l$ denote the element of $M_l$ induced by $G$. Then $\exp G^l \leq (\al_{r+1},j_{r+1})$. 

Suppose that
$\exp G^l < (\al_{r+1},j_{r+1})$. Then $\exp G^l \in \bigcup_{h=1}^r (\al_h,j_h)+\IN^n$, so that 
$\mon G^l = G^l_{(\al_h+\ga,j_h)}x^{(\al_h+\ga,j_h)}$, for some $h\leq r$ and $\ga\in\IN^n$, where 
$G^l_{(\al_h+\ga,j_h)} \in A_l$ is induced by an element of $A$, and $G^l_{(\al_h+\ga,j_h)}(a)=0$. On the 
other hand, $\mon G^l_h = G^l_{h, (\al_h,j_h)} x^{(\al_h,j_h)}$, where $G^l_{h, (\al_h,j_h)}\in A_l$ is induced
by an element of $A$ and $G^l_{h, (\al_h,j_h)}(a)\neq 0$. Let 
$H(x)= G^l_{h, (\al_h,j_h)} G^l(x) -  G^l_{(\al_h+\ga,j_h)} x^\ga G^l_h(x) \in A_l\llb x\rrb^p$. Then
$\exp G^l < \exp H \leq (\al_{r+1},j_{r+1})$ and the result follows.
\end{proof}

\subsection{Relations among parametrized families}\label{subsec:paramrel}
We continue to use the notation introduced in \S\ref{subsec:param}. For each $a\in X\backslash Z$,
consider the module of formal relations
\begin{equation}\label{eq:paramrel}
\Rel_a := \Rel(\Phi_1(a,x),\ldots,\Phi_q(a,x)) \subset \IR\llb x\rrb^q
\end{equation}
(see \S\ref{subsec:rel}). The following result is given in \cite[Ch.\,6]{BMPisa} in the case that
$\cQ=\cO$ and the germ of $X$ at $a_0$ is irreducible.

\begin{theorem}\label{thm:paramrel}
Let $a_0 \in X$. Then, after shrinking $U$ to a suitable neighbourhood of $a_0$,
there is a proper closed $\cQ$-subset $Y$ of $X$ containing $Z$, and a finite number of elements
$P_1,\ldots,P_s \in B\llb x\rrb^q$, where $B := \cQ(X,Y)$, such that $P_1(a,x),\ldots,P_s(a,x)$ 
generate $\Rel_a$, for all $a\in X\backslash Y$.
\end{theorem}

\begin{proof}
First suppose that the germ of $X$ at $a_0$ is irreducible.
By Corollary \ref{cor:irred}, we can assume that that $A$ is an integral domain.
Let $\IK$ denote the field of fractions of $A$.
Let $\Psi_1,\ldots,\Psi_r \in \IK\llb x\rrb^p$ denote the standard basis of $M$; in particular,
$(\al_i,j_i) := \exp \Psi_i$, $i=1,\ldots, r$, are the vertices of $\cN(M)$. By Corollary \ref{cor:prep},
we can assume that $\Phi_1,\ldots,\Phi_m$ and $\Psi_1,\ldots,\Psi_m$ (where $m\leq r,s$) are
smallest subsets of $\Phi_1,\ldots,\Phi_q$ and $\Psi_1,\ldots,\Psi_r$ (respectively) which generate
submodules with the same diagram $\cN(M)$.

(Regarding each $\Phi_i$ and $\Psi_i$ as a column vector with entries in $\IK\llb x\rrb$) write
\begin{align}
(\Phi_1 \cdots \Phi_q) &= (\Phi_1 \cdots \Phi_m)\cdot (\text{I}\ \Th),\label{eq:gen1}\\
(\Psi_1 \cdots \Psi_r) &= (\Psi_1 \cdots \Psi_m) \cdot (\text{I}\ \Xi)\label{eq:gen2},
\end{align}
where $\text{I}$ is the $m\times m$ identity matrix, and $\Th,\, \Xi$ are $m\times (q-m),\, m\times (r-m)$
matrices (respectively) with entries in $\IK\llb x\rrb$. By the formal division algorithm,
\begin{align*}
(\Phi_1 \cdots \Phi_m) &= (\Psi_1 \cdots \Psi_r)\cdot \text{T}\\
                                    &= (\Psi_1 \cdots \Psi_m) \cdot (\text{I}\ \Xi)\cdot \text{T},
\end{align*}
where $\text{T}$ is an $r\times m$ matrix with entries in $\IK\llb x\rrb$. In fact, there is a finitely
generated multiplicative subset $S$ of $A$ such that $\Psi_1,\ldots,\Psi_r \in S^{-1}A\llb x\rrb^p$
and $\Th,\, \Xi$ and $\text{T}$ have entries
in $S^{-1}A\llb x\rrb$ (see Remarks \ref{rem:mult,nak}).

Then $\text{U}(x) := ((\text{I}\ \Xi)\cdot \text{T})(x)$ is an $m\times m$ matrix with entries in $S^{-1}A\llb x\rrb$,
and $\text{U}(0)$ is invertible over $\IK$, by Corollary \ref{cor:prep}. Therefore, $\text{U}(x)$ is invertible
over $\IK\llb x\rrb$. Let $Y:= Z\cup W\cup W'$, where
\begin{align*}
W &:= \{a\in X\backslash Z:\, \text{ some generator of } S \text{ vanishes at } a\},\\
W' &:= \{a \in X\backslash (Z\cup W):\, \det \text{U}(a,0) = 0\}.
\end{align*}
Then $Y$ is a proper $\cQ$-subset of $X$ containing $Z$.

Note that \eqref{eq:gen1} induces an isomorphism
\begin{align*}
\Rel(\Phi_1,\ldots,\Phi_m) \oplus S^{-1}A\llb x\rrb^{q-m} &\to \Rel(\Phi_1,\ldots,\Phi_q)\\
(\xi,\,\zeta) &\mapsto (\xi - \Th\cdot \zeta,\, \zeta),
\end{align*}
where $\xi = (\xi_1,\ldots,\xi_m)$, $\zeta=(\zeta_1,\ldots,\zeta_{q-m})$. (Similarly, \eqref{eq:gen2}.)

Let $P_1,\ldots,P_s$ denote the ``standard relations'' among $\Psi_1,\ldots,\Psi_r$ (as given by
Theorem \ref{thm:rel}). Then, for all $a \in X\backslash(Z\cup W)$, $P_1(a,x),\ldots,P_s(a,x)$ are
the standard relations among $\Psi_1(a,x),\ldots,\Psi_r(a,x)$. Let $a \in X\backslash Y$. Then
\begin{enumerate}
\item $\Rel(\Psi_1(a,x),\ldots,\Psi_m(a,x))$ is generated by $\eta + \Xi(a,x)\cdot\zeta$, where
$(\eta,\zeta) = (\eta_1,\ldots,\eta_m,\zeta_1,\ldots,\zeta_{r-m})$ runs over $P_k(a,x)$, $k=1,\ldots,s$;

\smallskip
\item $\Rel(\Phi_1(a,x),\ldots,\Phi_m(a,x))$ is generated by $\text{U}^*(a,x)\cdot(\eta + \Xi(a,x)\cdot \zeta)$,
where $\text{U}^*$ denotes the adjoint matrix of $\text{U}$ and $(\eta,\zeta)$ runs over the $P_k(a,x)$.
\end{enumerate}
For each $k=1,\ldots,s$, put
$$
\xi_k(a,x) := \text{U}^*(a,x)\cdot(\eta + \Xi(a,x)\cdot \zeta),
$$
where $(\eta,\zeta) = P_k(a,x)$. Finally, then, $\Rel(\Phi_1(a,x),\ldots,\Phi_q(a,x))$ is generated by 
the relations $(\xi_k(a,x) - \Th(a,x)\cdot\zeta_l,\, \zeta_l)$, $k=1,\ldots,s$, $l=1,\ldots,q-m$,
where $\zeta_l = (0,\ldots,1,\ldots,0)$ with $1$ in the $l$th place. This completes the proof in the
case that the germ of $X$ at $a_0$ is irreducible.

For the general case, let us use the notation preceding Corollary \ref{cor:irredstrat}. Fix $l$ and consider
the proof above for the component $X_l$ of $X$ (writing $A_l$, $S_l$ and $M_l$ instead of $A$, $S$ and $M$,
and $\Psi^l_j$ instead of $\Psi_j$, $j=1,\ldots,r$, in the proof above). It follows from Corollary \ref{cor:irredstrat}
that we can assume that $\Psi^l_1,\ldots,\Psi^l_r$ are induced by elements of $A\llb x\rrb^p$, and that
the generators of $S_l$ are induced by elements of $A$. The result follows.
\end{proof}

\section{Proof of the division theorem}\label{sec:proof}
In this section, we will prove Theorem \ref{thm:main}. The proof will be by induction over a (local)
stratification of $U$ given by the following theorem, which summarizes
the results in Section \ref{sec:param}.

\begin{theorem}\label{thm:strat}
Let $\Phi_1,\ldots,\Phi_q \in \cQ(U)^p$, where $U$ is open in $\IR^n$. For all $a\in U$, let
$M_a \subset \IR\llb x\rrb^p$ denote the module generated by $\Phi(a,x) = \hPhi_{i,a}(x)$, 
$i=1,\ldots,q$, and let $\Rel_a \subset \IR\llb x\rrb^q$ denote the module of relations
$\Rel(\Phi_1(a,x),\ldots,\Phi_q(a,x))$ (see \S\ref{subsec:rel}). Then, given $a_0\in U$, there
is a neighbourhood $V$ of $a_0$ in $U$, and a finite filtration
\begin{equation}\label{eq:strat}
V = X_0 \supset X_1 \supset \cdots \supset X_{t+1} = \emptyset,
\end{equation}
such that, for each $k=0,\ldots,t$:
\begin{enumerate}
\item
$X_k$ is a closed $\cQ$-subset of $V$.
\smallskip
\item
$X_k\backslash X_{k+1}$ is smooth.
\smallskip
\item
The diagrams of initial exponents $\cN_a := \cN(M_a)$ and $\cN(\Rel_a)$ are constant (say, $\cN_a = \cN_k$
and $\cN(\Rel_a)=\cN(Rel)_k$) on $X_k\backslash X_{k+1}$.
\smallskip
\item
Let $A_k := \cQ(X_k,X_{k+1})$ (see \S\ref{subsec:geom}) and let $M_k$ denote
the submodule of $A_k\llb x\rrb^p$ generated by (the elements induced by) $\Phi_1,\ldots,\Phi_q$.
Then there exist $\Psi_{k1},\ldots,\Psi_{k,r_k} \in M_k$ such that $\Psi_{k1}(a,x),
\ldots,\Psi_{k,r_k}(a,x)$
represent the vertices of $\cN_a =\cN_k$, for all $a\in X_k\backslash X_{k+1}$.
\smallskip
\item
There exist $P_{k1},\ldots,P_{k,s_k} \in A_k\llb x\rrb^q$ such that $P_{k1}(a,x),\ldots,P_{k,s_k}(a,x)$ 
represent the vertices of $\cN(\Rel_a)=\cN(\Rel)_k$, for all $a\in X_k\backslash X_{k+1}$.
\end{enumerate}
\end{theorem}

The diagrams of initial exponents in Theorem \ref{thm:strat} are defined using the ordering of
$\IN^n\times\{1,\ldots,p\}$ (or of $\IN^n\times\{1,\ldots,q\}$) determined by any given positive linear
form $L$, as in Remark \ref{rem:order}. It is easy to see that Theorem \ref{thm:strat} follows from Theorems
\ref{thm:param}, \ref{thm:paramrel} and Corollary \ref{cor:irredstrat}, using the geometric lemmas
of \S\ref{subsec:geom}.

\subsection{Differential calculus lemmas}\label{subsec:diffcalc}
Let $U$ be an open subset of $\IR^n$.

\begin{lemma}[Borel's lemma]\label{lem:borel}
Let $S$ denote a closed $\cC^\infty$ submanifold
of $U$. Consider a field of formal power series on $S$,
$$
F(a,x) = \sum_{\al\in \IN^n} F_{\al}(a)x^\al \in \IR\llb x\rrb = \IR\llb x_1,\ldots,x_n\rrb,\quad a\in S
$$
(i.e., each $F_{\al} = F_{\al}(\xi)$ is a function on $S$). Then there exists $f\in \cC^\infty(U)$ such
that 
$$
F(a,x) = \hf_a(x) \in \cF_a = \IR\llb x\rrb, \quad a\in S,
$$
if and only if each $F_\al \in \cC^1(S)$ and, for all $a\in S$ and $v \in T_a S$ (where $T_a S$ 
denotes the tangent space of $S$ at $a$),
$$
\left(\p_{\xi,v}F\right)(a,x) = \left(\p_{x,v}F\right)(a,x),
$$
where $\left(\p_{\xi,v}F\right)(a,x)$ denotes the directional derivative of $F(\xi,x)$ with respect to $\xi$ at
$a$ in the direction $v$, and $\left(\p_{x,v}F\right)(a,x)$ denotes the formal directional derivative of $F(a,x)\in \IR\llb x\rrb$
in the direction $v$.
\end{lemma}

\begin{proof}
The lemma follows from the special case that $S$ is a linear subspace $\IR^k \times \{0\}$ of 
$\IR^n= \IR^k\times \IR^{n-k}$. In this case, the lemma is simply a reformulation of the following
classical statement of Borel's lemma: Given $g_\be(\xi) \in \cC^\infty(\IR^k)$, for all $\be \in \IN^{n-k}$,
there exists $g(\xi,\eta) \in \cC^\infty(\IR^n)$ (where $(\xi,\eta)=(\xi_1,\ldots,\xi_k,\eta_1,\ldots,\eta_{n-k})$)
such that $\p^{|\be|}g/\p \eta^\be = g_\be$, for all $\be$.
\end{proof}

The essential ingredient in the following two lemmas is {\L}ojasiewicz's inequalities, which hold for
functions of class $\cQ$ according to \cite[Thm.\,6.3]{BMselecta}.

\begin{lemma}[l'H\^opital--Hestenes lemma]\label{lem:hop}
Let $Z\subset X$ denote closed $\cQ$-subsets of $U$. Consider a field of formal power series
on $X$,
$$
F(a,x) = \sum_{\al\in \IN^n} F_{\al}(a)x^\al \in \IR\llb x\rrb,\quad a\in X.
$$
Assume that, 
\begin{enumerate}
\item
for each $\al$, $F_\al \in \cC^0(X)$ and $F_\al(a)=0$ if $a\in Z$;
\smallskip
\item
$F|_{X\backslash Z}$ is the field of Taylor series on $X\backslash Z$ of a $\cC^\infty$
function defined in a neighbourhood of $X\backslash Z$.
\end{enumerate}
Then $F$ is the field of Taylor series on $X$ of a $\cC^\infty$ function on $U$ (flat on $Z$).
\end{lemma}

\begin{proof}
According to \cite[Cor.\,8.2]{BMcomp}, \cite[Prop.\,3.4]{BMP} (generalizing \cite{Wh}), the assertion holds 
for any compact subsets $Z\subset X$ of $U$,
such that $X$ is $r$-\emph{regular}, for some positive integer $r$; i.e., with the property that there exists 
a constant $C$ such that any two points $a,b\in X$ can be joined by a rectifiable curve in $X$ of
length $\leq C|a-b|^{1/r}$. 
The regularity property for a suitable compact neighbourhood of any point
in a closed $\cQ$-subset $X$ of $U$ follows from resolution of singularities and {\L}ojasiewicz's inequalities 
(as proved in \cite[Thm.\,6.10]{BMihes} for the subanalytic case).
\end{proof}

\begin{lemma}[multiplier lemma]\label{lem:mult}
Let $X$ be a closed $\cQ$-subset of $U$. If $\varphi \in \cQ(X)$ and $f$ is the restriction to $X$
of a function $F\in \cC^\infty(U)$ which is flat on the zero-set $(\varphi =0)$, then $f/\varphi$ is the restriction to
$X$ of a $\cC^\infty$ function flat on $(\varphi =0)$.
\end{lemma}

\begin{proof}
By definition, $\varphi$ is the restriction to $X$ of a function $\Phi$ of class $\cQ$ defined in a neighbourhood
of $X$. Each partial derivative $(F/\Phi)^{(\al)} := (\p^{|\al|}/\p x^{\al})(F/\Phi)$ is the quotient of a $\cC^\infty$
function flat on $X\cap (\Phi = 0)$ by a power of $\Phi$. By Lemma  \ref{lem:hop}, it is enough 
to show that each $(F/\Phi)^{(\al)}|_X$ extends to a continuous function on $X$
which vanishes on $(\varphi =0)$. By resolution of singularities, we can assume that $X$ is smooth;
the result then follows directly from Lojasiewicz's inequality \cite[Thm.\,6.3\,III]{BMselecta};
cf. \cite[Ch.\,IV, Prop.\,1.4]{Malg}.
\end{proof}

\subsection{Proof of the Division Theorem \ref{thm:main}}\label{subsec:main}
Let $f\in \cC^\infty(U)^p$, and assume the equation \eqref{eq:main} admits a formal solution at every
point of $U$; i.e., for all $a\in U$, $\hf_a = \sum_{j=1}^q G_{j,a} \hPhi_{j,a}$, where each
$G_{j,a} \in \cF_a = \IR\llb x\rrb$. It is enough to find a $\cC^\infty$ solution of \eqref{eq:main} locally in $U$;
i.e., to show that, given $a_0 \in U$, there is a neighbourhood $V$ of $a_0$ and $g_1,\ldots,g_q \in \cC^\infty(V)$
such that $f = \sum_{j=1}^q g_j \Phi_j$ in $V$.

We can assume that $V$ admits a filtration \eqref{eq:strat} satisfying the conditions (1)--(5) of Theorem
\ref{thm:strat}. By induction over the filtration, it is enough to assume that $f$ is flat on $X_{k+1}$, for given
$k=0,\ldots,t$, and to show we can find $g_1,\ldots,g_q \in \cC^\infty(V)$ such that $f -  \sum_{j=1}^q g_j \Phi_j$
is flat on $X_k$. 

So let us write $X=X_k$, $Z=X_{k+1}$, and let us drop all subscripts $k$ in the statement of Theorem
\ref{thm:strat}. In other words, we write $\cN = \cN_k$, $\cN(\Rel) = \cN(\Rel)_k$, $A = A_k = \cQ(X,Z)$,
and we write $\Psi_1,\ldots,\Psi_r$ instead of $\Psi_{k1},\ldots,\Psi_{k,r_k}$, and $P_1,\ldots,P_s$
instead of $P_{k1},\ldots,P_{k,s_k}$.

Let $(\al_i,j_i)$, $i=1,\ldots,r$, denote the vertices of $\cN$, so we can assume that $\exp \Psi_i(a,x)
= (\al_i,j_i)$, $i=1,\ldots,r$, for all $a\in X\backslash Z$. By the formal division algorithm (Theorem \ref{thm:formaldiv})
and Corollary \ref{cor:diag} (with $\IR$ as the coefficient ring appearing in these results), 
for each $a\in X\backslash Z$, there is a unique expression
$$
\hf_a(x) = \sum_{i=1}^r \eta_i(a,x) \Psi_i(a,x),
$$
where
$$
\eta_i(a,x) = \sum_{(\al,j)\in \IN^n\times \{1,\ldots,p\}} \eta_{i,(\al,j)}(a) x^{(\al,j)},
$$
and $(\al_i,j_i) + \supp \eta_i(a,x)  \subset \De_i$, $i=1,\ldots,r$
(using the notation of \S\S\ref{subsec:formaldiv},\,\ref{subsec:diag}). 
Note that here we are using the formal division algorithm at each point $a\in X\backslash Z$ (the input for fixed
$a$ consists of formal power with coefficients in $\IR$). Each coefficient $\eta_{i,(\al,j)}$, as a function on
$X\backslash Z$, is the quotient of (the restriction to $X$ of) a $\cC^\infty$ function flat on $Z$ by a product
of powers of the initial coefficients $\Psi_{i,(\al_i,j_i)} \in \cQ(X,Z)$. It follows from the multiplier lemma 
\ref{lem:mult}, that every coefficient $\eta_{i,(\al,j)}$
is the restriction to $X\backslash Z$ of a $\cC^\infty$ function that is flat on $Z$.

By Theorem \ref{thm:strat}(4), re-expressing each $\Psi_i(a,x)$ in terms of $\Phi_{1,a}(x),\ldots,\Phi_{q,a}(x)$, we can write
\begin{equation}\label{eq:formaleq}
\hf_a(x) = \sum_{i=1}^q \zeta_i(a,x) \Phi_{i,a}(x),
\end{equation}
where each
$$
\zeta_i(a,x) = \sum_{(\al,j)\in \IN^n\times \{1,\ldots,p\}} \zeta_{i,(\al,j)}(a) x^{(\al,j)}
$$
and every coefficient $\zeta_{i,(\al,j)}(a)$, $a\in X\backslash Z$, is the restriction to $X\backslash Z$
of a $\cC^\infty$ function that is flat on $Z$. 

Again by the formal division algorithm (applied pointwise, with coefficients in $\IR$, as above), 
after dividing $\zeta(a,x) = (\zeta_1(a,x),\ldots,\allowbreak\zeta_q(a,x))$ by
the module of formal relations $\Rel_a$, $a\in X\backslash Z$, we can assume that, for all $a\in X\backslash Z$, 
$\supp \zeta(a,x)$ lies in the complement of $\cN(\Rel) \subset \IN^n\times\{1,\ldots,q\}$, and we maintain
the condition that each coefficient $\zeta_{i,(\al,j)}(a)$, $a\in X\backslash Z$, is the restriction to $X\backslash Z$
of a $\cC^\infty$ function that is flat on $Z$ (by Lemma \ref{lem:mult}).

Now, $S:=X\backslash Z$ is a $\cC^\infty$ manifold. By the l'H\^opital--Hestenes lemma \ref{lem:hop}, we need 
only show that each $\zeta_i(\xi,x)$, $\xi\in S$, is the field of Taylor expansions on $S$ of a $\cC^\infty$ function
defined in a neighbourhood of $S$. Consider $a\in S$ and $v\in T_a S$. Apply the operator
$\p_{\xi,v} - \p_{x,v}$ to the equation $F(\xi,x) = \sum_{i=1}^q \zeta_i(\xi,x) \Phi_i(\xi,x)$, where
$F(\xi,x) := \hf_\xi(x)$, $\Phi_i(\xi,x) := \Phi_{i,\xi}(x)$ (i.e., to the equation \eqref{eq:formaleq}). Since $f$ and all
$\Phi_i$ are $\cC^\infty$,
\begin{align*}
\left((\p_{\xi,v} - \p_{x,v})F\right)(a,x) &= 0,\\
\left((\p_{\xi,v} - \p_{x,v})\Phi_i\right)(a,x) &= 0,\quad i=1,\ldots,q.
\end{align*}
Therefore,
$$
0 = \sum_{i=1}^q \left((\p_{\xi,v} - \p_{x,v})\zeta_i\right)(a,x)\cdot \Phi_i(a,x);
$$
i.e., $\left((\p_{\xi,v} - \p_{x,v})\zeta\right)(a,x) \in \Rel_a$. But $\supp \left((\p_{\xi,v} - \p_{x,v})\zeta\right)(a,x)
\bigcap \cN(\Rel) = \emptyset$ (see Remark \ref{rem:diagdiff}). Therefore, for all $a\in S$,
$$
\left((\p_{\xi,v} - \p_{x,v})\zeta_i\right)(a,x) = 0,\quad i=1\ldots,q,
$$
and the result follows from Borel's lemma \ref{lem:borel}.\qed

\bibliographystyle{amsplain}

\end{document}